\documentclass{article}
\usepackage[utf8]{inputenc}
\usepackage{amsmath}
\usepackage{amssymb}
\usepackage{amsthm}
\usepackage{indentfirst}
\usepackage[shortlabels]{enumitem}
\usepackage{fullpage}
\usepackage{hyperref}
\usepackage{theoremref}
\usepackage{thmtools}
\usepackage{comment}
\usepackage{commath}
\usepackage{float}
\usepackage{listings}
\usepackage{graphicx}
\usepackage{biblatex}
\title{On the Sumset of Sets of Size $k$}
\author{Vincent Schinina}
\date{July 1, 2025}

\newcommand{\N}{\mathbb{N}}
\newcommand{\Z}{\mathbb{Z}}

\newcommand{\cc}[1]{\mathcal{#1}}
\newcommand{\pp}[1]{\left(#1\right)}
\newcommand{\bb}[1]{\left[#1\right]}
\newcommand{\br}[1]{\left\{#1\right\}}

\newtheorem{theorem}[subsection]{Theorem}
\newtheorem{definition}[subsection]{Definition}
\newtheorem{lemma}[subsection]{Lemma}
\newtheorem*{notation}{Notation}
\newtheorem{corollary}[subsection]{Corollary}
\allowdisplaybreaks
\begin{document}

\maketitle

\begin{abstract}
    The set $\mathcal{R}_{G}(h,k)$ consists of all possible sizes for the $h$-fold sumset of sets containing $k$ elements from an additive abelian group $G$. The exact makeup of this set is still unknown, but there has been progress towards determining which integers are present. We know that $\mathcal{R}_{G}(h,k)\subseteq\bb{hk-h+1,\binom{h+k-1}{h}}$, where the right side is an interval of integers that includes the endpoints. These endpoints are known to be attained. We will prove that the integers in $\bb{hk-h+2,hk-1}$ are not possible sizes for the $h$-fold sumset of a set containing $k\geq 4$ elements of a torsion-free additive abelian group $G$. Furthermore, we will confirm that this interval can't be made larger by exhibiting a subset of $G$ whose $h$-fold sumset has size $hk$. 
\end{abstract}

\begin{centering}
    \section*{Introduction}
\end{centering}
\addcontentsline{toc}{section}{Introduction}

In a paper by Nathanson \cite{nathanson2025problemsadditivenumbertheory}, he posed a problem, labeled Problem 1, about computing $\cc{R}_G(h,k)$ for an additive abelian group $G$, also defined below, for a fixed $k$ as $h$ increases. In Nathanson's Theorem 7, he confirmed that $hk-h+1$ is the minimum element of $\cc{R}_G(h,k)$ and that $hk-h+2$ is not in $\cc{R}_G(h,k)$, where $G$ is an ordered additive abelian group. Moreover, from Nathanson's Theorem 2, it follows that the maximum of $\cc{R}_G(h,k)$ is $\binom{h+k-1}{h}$, where $G$ is an ordered additive abelian group. In this paper, we will provide more insight about the structure of $\cc{R}_G(h,k)$ in a step towards solving Nathanson's Problem 1 for a torsion-free abelian group $G$. As motivation for the results, $G=\Z$ was examined first. Data retrieved from a computer program appeared to show that the result $hk-h+2$ is not in $\cc{R}_\Z(h,k)$ could be extended to the result that the inclusive interval of integers $\bb{hk-h+2,hk-1}$ is not in $\cc{R}_\Z(h,k)$, for all $h>1$ and $k\geq 4$. This was proven in the first version of this paper.\footnote{For an alternate method of proving this result, see the papers by Mohan and Pandey \cite{DBLP:journals/cdm/MohanP23}, and Tang and Xing \cite{Tang2021SOME}.} From that proof, it seemed that the result could be extended further to any torsion-free abelian group $G$ by using the same method. To do this, we will define all the necessary terminology, prove required facts, and then build up some useful preliminary results that will allow us to confirm this observation. Lastly, we will conclude that this interval can't be enlarged by showing that $hk\in\cc{R}_G(h,k)$.

\begin{centering}
    \section{Terminology}
\end{centering}

We'll begin by introducing definitions and notations that will be used throughout this paper.

\begin{notation}
    We denote some common number systems with the following symbols:
    \begin{align*}
        \text{Natural Numbers: }\N&=\br{1,2,3,4,5,\dots},\\
        \text{Natural Numbers with }0\text{: }\N_0&=\br{0,1,2,3,4,\dots},\\
        \text{Integers: }\Z&=\br{\dots,-2,-1,0,1,2,\dots}.
    \end{align*}
\end{notation}

\begin{notation}
    We write $\bb{a,b}$ to denote the inclusive interval of integers between $a$ and $b$. That is, $$[a,b]=\br{c\in\Z\mid a\leq c\leq b}.$$
\end{notation}

\begin{definition}
    A \textbf{group} is a binary operation $*$ on a set $G$ satisfying the following four properties:
    \begin{enumerate}[1)]
        \item{Closure: For all $a,b\in G$, $a*b\in G$,}
        \item{Associative: For all $a,b,c$, $\pp{a*b}*c=a*\pp{b*c}$,}
        \item{Identity: There exists $e\in G$ such that for all $a\in G$, $a*e=e*a=a$,}
        \item{Inverse: For all $a\in G$, there exists $a^{\prime}$ such that $a*a^{\prime}=a^{\prime}*a=e$.}
    \end{enumerate}
\end{definition}

\begin{definition}
    A group $G$ is \textbf{abelian} if it satisfies the commutative property: for all $a,b\in G$, $a*b=b*a$.
\end{definition}

\begin{definition}
    A relation $\preceq$ is a \textbf{total order} on a set $A$ if it satisfies the following four properties:
    \begin{enumerate}[1)]
        \item{Reflexive: For all $a\in A$, $a\preceq a$,}
        \item{Transitive: For all $a,b,c\in A$, if $a\preceq b$ and $b\preceq c$, then $a\preceq c$,}
        \item{Anti-Symmetric: For all $a,b\in A$, if $a\preceq b$ and $b\preceq a$, then $a=b$},
        \item{Strongly Connected:} For all $a,b\in A$, $a\preceq b$ or $b\preceq a$.
    \end{enumerate}
    We write $\prec$ instead of $\preceq$ when we want to exclude equality. Note that the transitive property holds in the stricter setting where $\prec$ replaces $\preceq$.
\end{definition}

\begin{definition}
    An group $G$ is \textbf{ordered} if it has a total order $\preceq$ satisfying $a*c\prec b*c$ and $c*a\prec c*b$, for all $a,b,c\in G$. Note that any subset $A$ of $G$ inherits this property.
\end{definition}

\begin{notation}
    For an ordered group $G$ with identity $e$, we use the following symbols to denote some important subsets of $G$:
    \begin{align*}
        \text{Positive Elements of }G\text{: }G^{+}&=\br{g\in G\mid e\prec g},\\
        \text{Nonnegative Elements of }G\text{: }G_e&=\br{g\in G\mid e\preceq g},\\
        \text{Negative Elements of }G\text{: }G^{-}&=\br{g\in G\mid g\prec e}.
    \end{align*}
\end{notation}

\begin{definition}
    A group $G$ is called $\textbf{additive}$ if it is written in additive notation. In this case, we use $+$ instead of $*$, $0$ instead of $e$, and $-g$ instead of $g^{\prime}$, for all $g\in G$. From the context, it will be clear when we are using the typical interpretation and the the additive group interpretation for $+, 0$, or $-$.
\end{definition}

\begin{notation}
    Let $g\in G$ and $h\in\Z$, where $G$ is an additive group. 
    \begin{enumerate}[1)]
    \item{If $h>0$, then $hg$ denotes the sum of $h$ many $g$'s. That is,
    $$hg=\underbrace{g+g+\cdots+g}_{h\text{ times}}.$$}
    \item{If $h=0$, then $hg=0$. The second $0$ is the identity element of $G$.}
    \item{If $h<0$, then $hg=-\bb{(-h)g}$, where $-\bb{(-h)g}$ is the inverse of $(-h)g$. Note that $-h>0$ which means that this requires (1) and can also be written as $hg=(-h)(-g)$, where $-g$ is the inverse of $g$.}
    \end{enumerate}
\end{notation}

\begin{definition}\thlabel{torsion-free}
    An additive group $G$ is \textbf{torsion-free} if for all $g\in G$ with $g\neq 0$, $hg\neq 0$, for all $h\in\N$.
\end{definition}

\begin{definition}\thlabel{sumset}
    Let $h\in\N$ with $h>1$. The \textbf{$h$-fold sumset} of a set $A$ of an additive abelian group $G$ is
    $$hA=\br{a_1+\cdots+a_h\mid a_1,\dots,a_h\in A}.$$
    Note that $a_1,\dots,a_h\in A$ do not need to be distinct. For a set $A=\br{a_1,\dots, a_k}\subseteq G$, we may also say
    $$hA=\br{b_1a_1+\cdots+b_ka_k\;\big|\; b_1,\dots,b_k\in\N_0\text{ and }b_1+\cdots+b_k=h}.$$
\end{definition}

\begin{notation}
    The set of $h$-fold sumset sizes for finite sets $A$ with size $k$ contained in an additive abelian group $G$ is denoted by
    $$\cc{R}_{G}(h,k)=\br{\abs{hA}\big| A\subseteq G,\;\abs{A}=k}.$$
\end{notation}

\begin{definition}\thlabel{trivial element}
    Let $A=\br{a_1,a_2,\dots,a_k}$ be a subset of an ordered additive abelian group $G$ with total order $\preceq$ such that $a_1\prec a_2\prec\cdots\prec a_k$. An element of $hA$ is called \textbf{trivial} if it appears on the increasing list (uses the ordered assumption) 
    \begin{align*}
        ha_1&\prec (h-1)a_1+a_2\prec\cdots\prec a_1+(h-1)a_2\prec ha_2\\
        &\prec(h-1)a_2+a_3\prec\cdots\prec a_2+(h-1)a_3\prec ha_3\\
        &\;\;\vdots\\
        &\prec (h-1)a_{k-1}+a_k\prec\cdots\prec a_{k-1}+(h-1)a_k\prec ha_k.
    \end{align*}
    Otherwise, the element is called \textbf{nontrivial}. Note that the above list contains $hk-h+1$ elements of $hA$.
\end{definition}

\begin{center}
    \section{Theorems about Ordered Additive Abelian Groups}
\end{center}

In the previous section, we defined an ordered additive abelian group and some notation. We will now prove some known facts about them. Note that the additive abelian group properties will be used without referencing them.

\begin{lemma}\thlabel{a<b and c<d implies a+c<b+d}
    Let $G$ be an ordered additive abelian group with respect to $\preceq$ and let $a,b,c,d\in G$. If $a\preceq b$ and $c\prec d$, then $a+c\prec b+d$.
\end{lemma}
\begin{proof}
    We begin by assuming that $a\preceq b$ and $c\prec d$. Using the fact that $G$ is ordered, we add $c$ to both sides of $a\preceq b$ (doesn't require the ordered assumption if equality occurs) and $b$ to both sides of $c\prec d$. This gives us
    $$a+c\preceq b+c\quad\text{and}\quad b+c\prec b+d,$$
    respectively. By the transitive property (or substitution if equality occurs), we see that $a+c\prec b+d$, as wanted.
\end{proof}

\begin{lemma}\thlabel{a<b iff a+c=b}
    Let $G$ be an ordered additive abelian group with respect to $\preceq$ and let $a,b\in G$. Then, $a\prec b$ if and only if there exists a unique $c\in G^{+}$ such that $a+c=b$.
\end{lemma}
\begin{proof}
    We start by assuming that $a\prec b$. For existence, we use the fact that $G$ is ordered and add $-a$ to both sides. With simplifications, we obtain $0\prec -a+b$. By definition, $-a+b\in G^{+}$. So, let $c=-a+b$. It follows that $a+c=b$, as desired.

    For uniqueness, we assume that, there also exists $c^{\prime}$ such that $a+c^{\prime}=b$. By substitution, $a+c=a+c^{\prime}$. By adding $-a$ to both sides and simplifying, we see that $c=c^{\prime}$. Hence, $c\in G^{+}$ is unique.
    
    Conversely, suppose there exists a unique $c\in G^{+}$ such that $a+c=b$. Since $c\in G^{+}$, then, by definition, $0\prec c$. Using the fact that $G$ is ordered, we can add $a$ to both sides and, with simplifications, get $a\prec a+c$. By substitution of $a+c=b$, we have that $a\prec b$.
\end{proof}

\begin{lemma}\thlabel{g>0 iff -g<0}
    Let $G$ be an ordered additive abelian group with respect to $\preceq$ and let $g\in G$ with $g\neq 0$. Then, $g\in G^{+}$ if and only if $-g\in G^{-}$.
\end{lemma}
\begin{proof}
    We begin by assuming that $g\in G^{+}$. By definition, $0\prec g$. Using the fact that $G$ is ordered, we can add $-g$ to both sides and, with simplifications, obtain $-g\prec 0$. Hence, $-g\in G^{-}$, by definition. 

    Lastly, assume $-g\in G^{-}$. By definition, $-g\prec 0$. Using the fact that $G$ is ordered, we can add $g$ to both sides and, with simplifications, obtain $0\prec g$. By definition, $g\in G^{+}$.
\end{proof}

\begin{lemma}\thlabel{0<g implies 0<hg}
    Let $G$ be an ordered additive abelian group with respect to $\preceq$. If $g\in G^{+}$, then, for all $h\in\N$, $hg\in G^{+}$. 
\end{lemma}
\begin{proof}
    If $h=1$, then this is immediate. Take $h=2$. Since $g\in G^{+}$, then $0\prec g$. Using the fact that $G$ is ordered, we have that $g\prec g+g=2g$. Hence, by the transitive property, $0\prec 2g$. By induction on $h$, it will follow that $0\prec hg$, for all $h\in\N$. Therefore, $hg\in G^{+}$.
\end{proof}

\begin{lemma}\thlabel{g<0 implies hg<0}
    Let $G$ be an ordered additive abelian group with respect to $\preceq$. If $g\in G^{-}$, then, for all $h\in\N$, $hg\in G^{-}$. 
\end{lemma}
\begin{proof}
    If $h=1$, then this is immediate. Take $h=2$. Since $g\in G^{-}$, then $g\prec 0$. Using the fact that $G$ is ordered, we have that $2g=g+g\prec g$. Hence, by the transitive property, $2g\prec 0$. By induction on $h$, it will follow that $hg\prec 0$, for all $h\in\N$. Therefore, $hg\in G^{-}$.
\end{proof}

\begin{lemma}\thlabel{ha>0 iff h>0}
    Let $G$ be an ordered additive abelian group with respect to $\preceq$ and let $a\in G^{+}$. For $h\in\Z$, $ha\in G^{+}$ if and only if $h>0$.
\end{lemma}
\begin{proof}
    For the forward direction, we prove the contrapositive. So, we assume $h\leq 0$ and show that $ha\notin G^{+}$. If $h=0$, then $ha=0$. Since $0\notin G^{+}$, then $ha\notin G^{+}$. If $h<0$, then our notation says that $ha=(-h)(-a)$ and $-h>0$. Since $a\in G^{+}$, then, by \thref{g>0 iff -g<0}, $-a\in G^{-}$. Therefore, by \thref{g<0 implies hg<0}, we have that $ha=(-h)(-a)\in G^{-}$, which implies that $ha\notin G^{+}$. This confirms the contrapositive.

    Conversely, we now assume that $h>0$. By \thref{0<g implies 0<hg}, we have that $ha\in G^{+}$, as needed.
\end{proof}

\begin{theorem}\thlabel{ordered implies torsion-free}
    Every ordered additive abelian group is torsion-free. 
\end{theorem}
\begin{proof}
    Let $G$ be an ordered additive abelian group with respect to $\preceq$. Let $g\in G$ such that $g\neq 0$. Then, $g\in G^{+}$ or $g\in G^{-}$. By \thref{0<g implies 0<hg} and \thref{g<0 implies hg<0}, respectively, we have that $hg\in G^{+}$ or $hg\in G^{-}$, for all $h\in\N$. In either case, we see that $hg\neq 0$, for all $h\in\N$. Since $g\in G$ with $g\neq 0$ was an arbitrary choice, then it follows, by \thref{torsion-free}, that $G$ is torsion-free.
\end{proof}

\begin{corollary}\thlabel{hg=0 implies h=0}
    Let $G$ be an ordered additive abelian group with respect to $\preceq$ and let $g\in G$ with $g\neq 0$. If $hg=0$, for some $h\in\Z$, then $h=0$.
\end{corollary}
\begin{proof}
    First, $hg=0$, when $h=0$, by our notation. We are tasked to show that this is the only solution. Since $G$ be an ordered abelian group, then, by \thref{ordered implies torsion-free}, $G$ is torsion-free. By assumption, $g\neq 0$, which means $-g\neq 0$. Hence, for all $f\in\N$, $fg,f(-g)\neq 0$. From our notation, $f(-g)=(-f)g$, where $-f<0$. So, for all $h\in\Z$ with $h\neq 0$, $hg\neq 0$. Thus, it must be the case that if $hg=0$, then $h=0$.
\end{proof}

\begin{lemma}\thlabel{h<=i iff ha<=ia}
    Let $G$ be an ordered additive abelian group with respect to $\preceq$ and let $a\in G^{+}$. For $h,i\in\Z$, $h\leq i$ if and only if $ha\preceq ia$. Specifically, $h<i$ if and only if $ha\prec ia$ and $h=i$ if and only $ha=ia$.
\end{lemma}
\begin{proof}
    We start by assuming $h\leq i$. If $h\neq i$, then, we have that $i-h>0$. By \thref{ha>0 iff h>0}, it follows that $(i-h)a\in G^{+}$. By definition, this mean $0\prec(i-h)a$. Adding $ha$ to both sides and simplifying, we have that $ha\prec ia$. Otherwise, $h=i$ and we immediately have that $ha=ia$. Hence, we have that $ha\preceq ia$.

    Next, assume that $ha\preceq ia$. Adding $-ha$ to both sides and simplifying gives us $0\preceq (i-h)a$. If $0\prec (i-h)a$, then $(i-h)a\in G^{+}$. By \thref{ha>0 iff h>0}, we get that $i-h>0$. Hence, we can conclude that $h<i$. Otherwise, $(i-h)a=0$. By \thref{hg=0 implies h=0}, $i-h=0$, which means $h=i$. So, we can conclude that $h\leq i$.
\end{proof}

\begin{theorem}\thlabel{torsion-free abelian implies ordered}
    Every torsion-free additive abelian group can be ordered.
\end{theorem}
\begin{proof}
    See \cite{levi1942ordered} and \cite{levi1943contributions}.
\end{proof}

\begin{centering}
    \section{A Missing Interval from \texorpdfstring{$\cc{R}_{G}(h,k)$}{R(h,k)}}
\end{centering}

Before we can prove that the integers in $\bb{hk-h+2,hk-1}$ are missing from $\cc{R}_{G}(h,k)$, we need to prove a known fact about $h$-fold sumsets in an additive abelian group and then we will need to confirm some other useful preliminary results:
\begin{lemma}\thlabel{sumset size is translation invariant}
    Let $A=\br{a_1,a_2,\dots,a_k}$ and $B=\br{a_1+b,a_2+b,\dots, a_k+b}$ be subsets of an additive abelian group $G$. Note that $B$ is a translation of $A$ by $b$. Then, $\abs{hA}=\abs{hB}$. In other words, the size of a sumset is translation invariant.
\end{lemma}
\begin{proof}
    To show that $\abs{hA}=\abs{hB}$, we will prove that $hB=\br{a+hb\mid a\in hA}$. First, we will confirm that $hB\subseteq\br{a+hb\mid a\in hA}$. Let $x\in hB$. By \thref{sumset}, 
    $$x=c_1\pp{a_1+b}+\cdots+c_k\pp{a_k+b},$$
    for some $c_1,\dots,c_k\in\N_0$ with $c_1+\cdots+c_k=h$. Using our notation and the abelian group properties, we can rewrite this as follows:
    $$x=\pp{c_1a_1+\cdots+c_ka_k}+\pp{c_1b+\cdots+c_kb}.$$
    Since $c_1+\cdots+c_k=h$ and $A=\br{a_1,a_2,\dots,a_k}$, then the first summand is $h$-fold sum of elements of $A$ and the second summand is just $b$ added to itself $h$ times. So, $a=c_1a_1+\cdots+c_ka_k\in hA$, by \thref{sumset}. Hence, we have that
    $$x=a+hb,$$
    which means that $x\in\br{a+hb\mid a\in hA}$. Therefore, we get that $hB\subseteq\br{a+hb\mid a\in hA}$.

    Next, we will show that $\br{a+hb\mid a\in hA}\subseteq hB$. Let $x\in\br{a+hb\mid a\in hA}$. By definition, this means that $x=a+hb$, for some $a\in hA$. Since $a\in hA$, then, by \thref{sumset},
    $$a=c_1a_1+\cdots+c_ka_k,$$
    for some $c_1,\dots,c_k\in\N_0$ with $c_1+\cdots+c_k=h$. By substitution, this leaves us with
    $$x=\pp{c_1a_1+\cdots+c_ka_k}+\pp{c_1+\cdots+c_k}b.$$
    Using our notation and the abelian group properties, we can rewrite this as follows:
    $$x=c_1\pp{a_1+b}+\cdots+c_k\pp{a_k+b}.$$
    Since $B=\br{a_1+b,a_2+b,\dots, a_k+b}$, then $x$ is an $h$-fold sum of elements in $B$, which means $x\in hB$, by \thref{sumset}. Hence, $\br{a+hb\mid a\in hA}\subseteq hB$.

    With this, we can now conclude that $hB=\br{a+hb\mid a\in hA}$. Observe that the set on the right side is the same size as $hA$. Thus, it follows that $\abs{hA}=\abs{hB}$.
\end{proof}

\begin{lemma}\thlabel{nontrivial in hB implies nontrivial in hA}
    Let $A=\br{a_1,a_2,\dots,a_k}$ and $B=\br{a_1,a_2,\dots,a_{k-1}}$ be subsets of an ordered additive abelian group $G$, for some $k\in\N$ with $k>1$ and $a_1\prec\cdots\prec a_k$. Note that this makes $B\subseteq A$. Then, $hB\subseteq hA$. Furthermore, if $b$ is a nontrivial element of $hB$, then $b$ is also a nontrivial element of $hA$. 
\end{lemma}
\begin{proof}
    Let $b\in hB$. Then, by \thref{sumset},
    $$b=c_1a_1+c_2a_2+\dots+c_{k-1}a_{k-1},$$
    where $c_1,\dots, c_{k-1}\in\N_0$ and $c_1+\cdots+c_{k-1}=h$. Since $B\subseteq A$, then this also satisfies the definition of $hA$. Hence, $b\in hA$, and we have that $hB\subseteq hA$.

    To prove the last statement, we'll prove its contrapositive. Let $b\in hB$. Suppose that $b$ is a trivial element of $hA$, which makes sense because $hB\subseteq hA$, as we just showed. Then, by \thref{trivial element},
    $$b=(h-i)a_j+ia_{j+1},$$
    for some $0\leq i\leq h$ and $1\leq j\leq k-1$. If $j=k-1$ and $1\leq i\leq h$, then, using the ordered property and \thref{h<=i iff ha<=ia}, 
    $$ha_{k-1}=(h-i)a_{k-1}+ia_{k-1}\prec(h-i)a_{k-1}+ia_{k}=b,$$
    where strict inequality is always ensured because $i\neq 0$ and $a_{k-1}\prec a_k$. Since the largest element of $hB$ is $ha_{k-1}$, then this implies $b\not\in hB$, which is a contradiction. Otherwise, $b$ is a trivial element of $hB$, by \thref{trivial element}, which proves the contrapositive. Therefore, it follows that if $b$ is a nontrivial element of $hB$, then $b$ is also a nontrivial element of $hA$. 
\end{proof}

\begin{theorem}\thlabel{1}
    Let $k\geq 4$ and let $G$ be an ordered additive abelian group with respect to $\preceq$. Suppose $0\prec a_1\prec a_2\prec\cdots\prec a_{k-3}$ for $a_{i}\in G^{+}$. If 
    $$A=\br{0,a_1,a_2,\dots, a_{k-3}, a_{k-3}+b, a_{k-3}+b+c},$$ 
    for some $b,c\in G^{+}$, and $c\neq db$, for all $d\in\N$, then $hA$ has at least $h-1$ nontrivial elements.
\end{theorem}
\begin{proof}
    We can assume that $A=\br{0,a_1,a_2,\dots, a_{k-3}, a_{k-3}+b, a_{k-3}+b+c}$, for some $b,c\in G^{+}$, and $c\neq db$, for all $d\in\N$. As in \thref{trivial element}, we can write out the $hk-h+1$ trivial elements of $hA$. With simplifications, we are left with
    \begin{equation}\label{increasing elements of hA}
        \begin{aligned}
            0&\prec a_1\prec\cdots\prec (h-1)a_1\prec ha_1\\
            &\prec(h-1)a_1+a_2\prec\cdots<a_1+(h-1)a_2\prec ha_2\\
            &\;\;\vdots\\
            &\prec ha_{k-3}+b\prec\cdots\prec ha_{k-3}+(h-1)b\prec h\pp{a_{k-3}+b}\\
            &\prec h\pp{a_{k-3}+b}+c\prec\cdots\prec h\pp{a_{k-3}+b}+(h-1)c\prec h\pp{a_{k-3}+b+c}.
        \end{aligned}
    \end{equation}
    Observe that for all $e\in\bb{1,h-1}$, 
    $$ha_{k-3}\prec ha_{k-3}+eb+c\prec h\pp{a_{k-3}+b}+c$$
    $$\text{and}$$
    $$ha_{k-3}+eb+c=(h-e)a_{k-3}+(e-1)\pp{a_{k-3}+b}+\pp{a_{k-3}+b+c}.$$
    Hence, $ha_{k-3}+eb+c\in hA$, for all $e\in\bb{1,h-1}$, and it fits somewhere in the second to last line of inequalities in \eqref{increasing elements of hA}. These are $h-1$ nontrivial elements of $hA$, otherwise
    $$ha_{k-3}+eb+c=ha_{k-3}+fb,$$
    for some $f\in\bb{1,h}$, will contradict, using \thref{ha>0 iff h>0}, the assumption that $c\neq db$, for all $d\in\N$.
\end{proof}

\begin{theorem}\thlabel{2}
    Let $k\geq 4$ and let $G$ be an ordered additive abelian group with respect to $\preceq$. Suppose $0\prec a_1\prec a_2\prec\cdots\prec a_{k-3}$ for $a_{i}\in G^+$. If 
    $$A=\br{0,a_1,a_2,\dots, a_{k-3}, a_{k-3}+b, a_{k-3}+b+c}$$ 
    and $c=db$, for some $b,c\in G^{+}$ and $d\in\N$ with $d>1$, then $hA$ has at least $h-1$ nontrivial elements.
\end{theorem}
\begin{proof}
    We can assume that $A=\br{0,a_1,a_2,\dots, a_{k-3}, a_{k-3}+b, a_{k-3}+b+c}$ and $c=db$, for some $b,c\in G^{+}$ and $d\in\N$ with $d>1$. As in \thref{trivial element}, we can write out the $hk-h+1$ trivial elements of $hA$ using $c=db$. With simplifications, we have
    \begin{equation}\label{increasing elements with c=db}
        \begin{aligned}
            0&\prec a_1\prec\cdots\prec (h-1)a_1\prec ha_1\\
            &\prec (h-1)a_1+a_2\prec\cdots\prec a_1+(h-1)a_2\prec ha_2\\
            &\;\;\vdots\\
            &\prec ha_{k-3}+b\prec\cdots\prec ha_{k-3}+(h-1)b\prec h\pp{a_{k-3}+b}\\
            &\prec h\pp{a_{k-3}+b}+db\prec\cdots\prec h\pp{a_{k-3}+b}+(h-1)db\prec h\pp{a_{k-3}+b+db}.
        \end{aligned}
    \end{equation}
    Observe that for all $e\in\bb{1,h-1}$,
    $$h\pp{a_{k-3}+b}\prec h\pp{a_{k-3}+b}+(ed-1)b\prec h\pp{a_{k-3}+b+db},$$
    where the first inequality needs the fact that $d>1$. Notice that $h\pp{a_{k-3}+b}+(ed-1)b\in hA$, for all $e\in\bb{1,h-1}$ because
    $$h\pp{a_{k-3}+b}+(ed-1)b=a_{k-3}+(h-e-1)\pp{a_{k-3}+b}+e\pp{a_{k-3}+b+db}\in hA,$$
    where we used $db$ in place of $c$ as allowed by the assumption that $c=db$. 
    
    From the inequality, these fit somewhere in the last line of \eqref{increasing elements with c=db}. Therefore, these are $h-1$ nontrivial elements of $hA$, otherwise 
    $$h\pp{a_{k-3}+b}+(ed-1)b=h\pp{a_{k-3}+b}+fdb,$$
    for some $f\in\bb{1,h-1}$, will imply, using \thref{hg=0 implies h=0}, that $d=1$, which contradicts $d>1$.
\end{proof}

\begin{theorem}\thlabel{existence of h-1 nontrivial elements in hB}
     Let $G$ be an ordered additive abelian group with respect to $\preceq$. For every $k\geq 4$, if $A=\br{0,a,a+b_1,\dots, a+b_1+\cdots+b_{k-2}}\subseteq G$ with $a,b_1,\dots, b_{k-2}\in G^{+}$ and $b_{k-2}=b_{k-3}$, then either $A$ is a $k$-term arithmetic progression or the set $B=\br{0,a,a+b_1,\dots, a+b_1+\cdots+b_{k-3}}\subseteq G$ produces the sumset $hB$ that has at least $h-1$ nontrivial elements.
\end{theorem}
\begin{proof}
    We'll proceed by induction on $k$.
    
    \noindent\textbf{\underline{Base Case:}} Suppose $k=4$.

    Then, we assume $A=\br{0,a,a+b_1, a+b_1+b_2}$ with $a,b_1,b_2\in G^{+}$ and $b_2=b_1$. This means we are working with $A=\br{0,a, a+b_1, a+2b_1}$ and $B=\br{0,a,a+b_1}$. We can list out the $3h+1$ elements of $hB$ as seen in \thref{trivial element}. With simplifications, we obtain
    \begin{equation}\label{increasing elements with b_2=b_1}
        \begin{aligned}
            0&\prec a\prec 2a\prec\cdots\prec ha\\
            &\prec ha+b_1\prec ha+2b_1\prec\cdots\prec ha+(h-1)b_1\prec h\pp{a+b_1}.\\
        \end{aligned}
    \end{equation}
    Now, we will split this into two cases:

    \noindent\textbf{Case 1:} Suppose $b_1\neq ea$, for all $e\in\N$.

    Observe that for all $f\in\bb{1,h-1}$,
    $$a\prec fa+b_1\prec ha+b_1\quad\text{and}\quad fa+b_1=(h-f)0+(f-1)a+(a+b_1).$$
    So, we have that $fa+b_1\in hB$ and fits in the first line of \eqref{increasing elements with b_2=b_1}. These are $h-1$ nontrivial elements of $hB$, otherwise 
    $$fa+b_1=ga,$$
    for some $g\in\bb{2,h}$, will contradict, using \thref{ha>0 iff h>0}, the Case 1 assumption.

    \noindent\textbf{Case 2:} Suppose $b_1=ea$, for some $e\in\N$.

    If $e=1$, then $b_1=a$, and $A=\br{0,a,a+b_1,a+2b_1}$ becomes $A=\br{0,a,2a,3a}$, which means $A$ is a $4$-term arithmetic progression.

    Otherwise, $e>1$, which means $A=\br{0,a,a+b_1,a+2b_1}$ becomes $A=\br{0,a,(e+1)a,(2e+1)a}$, and $B=\br{0,a,(e+1)a}$. Rewriting \eqref{increasing elements with b_2=b_1}, we have
    \begin{equation}\label{increasing elements with b_2=b_1 and e>1}
        \begin{aligned}
            0&\prec a\prec 2a\prec\cdots\prec ha\\
            &\prec (h+e)a\prec (h+2e)a\prec\cdots\prec\bb{h+(h-1)e}a\prec (h+he)a.\\
        \end{aligned}
    \end{equation}
    Observe that for all $f\in\bb{1,h-1}$,
    $$ha\prec\pp{h+ef-1}a\prec(h+he)a,$$
    where the first inequality needs the assumption that $e>1$. It follows that $\pp{h+ef-1}a\in hB$, for all $f\in\bb{1,h-1}$ because
    $$\pp{h+ef-1}a=0+(h-f-1)a+f(e+1)a\in hB.$$
    From the inequality, we know that these terms fit somewhere in the second line of \eqref{increasing elements with b_2=b_1 and e>1}. These are $h-1$ nontrivial elements of $hB$, otherwise 
    $$\pp{h+ef-1}a=(h+ge)a,$$ 
    for some $g\in\bb{1,h-1}$, will imply, using \thref{hg=0 implies h=0}, that $e=1$, which contradicts $e>1$. 

    At this point, we exhausted all necessary cases. Collecting all the results, we see that either $A$ is a $4$-term arithmetic progression or $hB$ has at least $h-1$ nontrivial elements. Thus, the base case holds.

    \noindent\textbf{\underline{Inductive Hypothesis:}} Assume that if $A=\br{0,a,a+b_1,\dots, a+b_1+\cdots+b_{k-2}}$ with \\$a,b_1,\dots, b_{k-2}\in G^{+}$ and $b_{k-2}=b_{k-3}$, then either $A$ is $k$-term arithmetic progression or the set \\$B=\br{0,a,a+b_1,\dots, a+b_1+\cdots+b_{k-3}}$ produces a sumset $hB$ that has at least $h-1$ nontrivial elements.

    \noindent\textbf{\underline{Inductive Step:}} Show that if $A=\br{0,a,a+b_1,\dots, a+b_1+\cdots+b_{k-1}}$ with \\$a,b_1,\dots, b_{k-1}\in G^{+}$ and $b_{k-1}=b_{k-2}$, then either $A$ is a $(k+1)$-term arithmetic progression or the set \\$B=\br{0,a,a+b_1,\dots, a+b_1+\cdots+b_{k-2}}$ produces a sumset $hB$ that has at least $h-1$ nontrivial elements.

    We start by assuming $A=\br{0,a,a+b_1,\dots, a+b_1+\cdots+b_{k-1}}$ with $a,b_1,\dots, b_{k-1}\in G^{+}$ and $b_{k-1}=b_{k-2}$. Then, we are working with
    \begin{align*}
        A&=\br{0,a,a+b_1,\dots, a+b_1+\cdots+b_{k-2}, a+b_1+\cdots+2b_{k-2}},\\
        B&=\br{0,a,a+b_1,\dots, a+b_1+\cdots+b_{k-2}}.
    \end{align*}
    
    \noindent\textbf{Case 1:} Suppose $b_{k-2}\neq eb_{k-3}$, for all $e\in\N$.

    By \thref{1} on $B$, we have that there are $h-1$ nontrivial elements in $hB$.

    \noindent\textbf{Case 2:} Suppose $b_{k-2}=eb_{k-3}$, for some $e\in\N$.
    
    If $e>1$, then, by \thref{2} on $B$, there are $h-1$ nontrivial elements in $hB$. 
    
    Otherwise, $e=1$ and this means $b_{k-2}=b_{k-3}$. Therefore, by the Inductive Hypothesis, either $B$ is a $k$-term arithmetic progression or the set $C=\br{0,a,a+b_1,\dots, a+b_1+\cdots+b_{k-3}}$ produces the sumset $hC$ that has $h-1$ nontrivial elements. In the case that $B$ is a $k$-term arithmetic progression, then it follows that $b_i=a$, for all $1\leq i\leq k-2$.
    Substituting $b_i=a$ into the elements in $A$, for all $1\leq i\leq k-2$, gives us
    $$A=\br{0,a,2a,\dots,(k-1)a,ka},$$
    which indicates that $A$ is a $(k+1)$-term arithmetic progression. In the case that $hC$ has $h-1$ nontrivial elements, then, by \thref{nontrivial in hB implies nontrivial in hA}, these elements are also nontrivial in $hB$.
    
    At this point, we considered all necessary cases and it can be seen that either $A$ is a $(k+1)$-term arithmetic progression or $hB$ has at least $h-1$ nontrivial elements. This completes the inductive step.

    Thus, we may now conclude that for every $k\geq 4$, if $A=\br{0,a,a+b_1,\dots, a+b_1+\cdots+b_{k-2}}$ with $a,b_1,\dots, b_{k-2}\in\N$ and $b_{k-2}=b_{k-3}$, then either $A$ is a $k$-term arithmetic progression or the set \\$B=\br{0,a,a+b_1,\dots, a+b_1+\cdots+b_{k-3}}$ produces the sumset $hB$ that has at least $h-1$ nontrivial elements.
\end{proof}

The next theorem is the main result. Most of the heavy lifting occurred in the preceding lemmas and theorems.
\begin{theorem}
    Let $G$ be a torsion-free additive abelian group. For all $k\geq 4$, $\mathcal{R}_{G}(h,k)\cap\bb{hk-h+2,hk-1}=\varnothing$.
\end{theorem}
\begin{proof}
    Let $k\geq 4$ be given. By \thref{torsion-free abelian implies ordered}, $G$ can be ordered and we'll say that it is ordered with respect to $\preceq$. Since \thref{sumset size is translation invariant} confirmed that a translation doesn't change the size of the sumset, it suffices to consider the set $A=\br{0,a_1,\dots,a_{k-1}}\subseteq G$ with $0\prec a_1\prec\cdots\prec a_{k-1}$. Based on this inequality, we may apply \thref{a<b iff a+c=b} to write
    $$a_{k-2}=a_{k-3}+b\quad\text{and}\quad a_{k-1}=a_{k-3}+b+c$$
    for some $b,c\in G^{+}$. This leaves us with 
    $$A=\br{0,a_1,a_2,\dots, a_{k-3},a_{k-3}+b,a_{k-3}+b+c}.$$ 
    We will need two cases:

    \noindent\textbf{Case 1:} Suppose that $c\neq db$, for all $d\in\N$.

    By \thref{1}, it follows that $hA$ has $h-1$ nontrivial elements. As noted in \thref{trivial element}, $hA$ has $hk-h+1$ trivial elements. So, we must have that
    $$\abs{hA}\geq (hk-h+1)+(h-1)=hk.$$

    \noindent\textbf{Case 2:} Suppose that $c=db$, for all $d\in\N$.

    If $d>1$, then, by \thref{2}, we know that $hA$ has $h-1$ nontrivial elements. As noted in \thref{trivial element}, $hA$ has $hk-h+1$ trivial elements. So, we must have that
    $$\abs{hA}\geq (hk-h+1)+(h-1)=hk.$$
    
    If $d=1$, then we have that $c=b$. By \thref{existence of h-1 nontrivial elements in hB}, we know that either $A$ is a $k$-term arithmetic progression or that the set
    $$B=\br{0,a_1,a_2,\dots, a_{k-3}, a_{k-3}+b}$$
    produces the sumset $hB$ with $h-1$ nontrivial elements. In the case that $A$ is a $k$-term arithmetic progression, then we know from Theorem 2 in Nathanson's paper \cite{nathanson2025problemsadditivenumbertheory} that $\abs{hA}=hk-h+1$. In the case that $hB$ has at least $h-1$ nontrivial elements, it follows by \thref{nontrivial in hB implies nontrivial in hA} that $hA$ has at least $h-1$ nontrivial elements. As noted in \thref{trivial element}, $hA$ has $hk-h+1$ trivial elements. Therefore, we must have that
    $$\abs{hA}\geq (hk-h+1)+(h-1)=hk.$$

    At this point, we considered all the required cases and saw that either
    $$\abs{hA}=hk-h+1\quad\text{or}\quad\abs{hA}\geq hk.$$
    Since $k\geq 4$ was arbitrary, then this is true for any $k\geq 4$. Thus, for all $k\geq 4$,
    $$\mathcal{R}_{G}(h,k)\cap\bb{hk-h+2,hk-1}=\varnothing,$$
    as desired.
\end{proof}

An immediate question to ask is: can this interval be enlarged? We know from Theorem 2 in Nathanson's paper \cite{nathanson2025problemsadditivenumbertheory} that a $k$-term arithmetic progression produces an $h$-fold sumset with size $hk-h+1$. So, the interval can't be extended to the left. As for extending it to the right, we prove the following two results:

\begin{theorem}\thlabel{A exists with |hA|=hk}
    Let $G$ be an ordered additive abelian group with respect to $\preceq$ and let $a\in G$ such that $a\neq 0$. From \thref{g>0 iff -g<0}, we can take $a\in G^{+}$. For $k\geq 3$, if $A=\br{0,a,2a,\dots, (k-2)a,ka}$, then $\abs{hA}=hk$. 
\end{theorem}
\begin{proof}
    Let $k\geq 3$ be given and assume that $A=\br{0,a,2a,\dots, (k-2)a,ka}$. We will prove that $hA=\br{ja\;\big|\; j\in[0,hk]-\br{hk-1}}$. For all $c\in hA$ with $c\neq h(ka)$,
    $$c=b_0(0)+b_1a+b_2(2a)+\cdots+b_{k-2}\bb{(k-2)a}+b_{k-1}(ka),$$
    where $b_0,b_1,b_2,\dots, b_{k-1}\in\N_0$ with $b_0+\cdots+b_{k-1}=h$, by \thref{sumset}. Since $c\neq h(ka)$, then $b_i$, for some $0\leq i\leq k-2$, must be nonzero. Since $a\in G^{+}$, then, by \thref{h<=i iff ha<=ia}, $0\prec a\prec 2a\prec\cdots\prec(k-2)a\prec ka$. If $b_{k-1}=h-1$ and $b_{k-2}=1$, then 
    $$c=(k-2)a+(h-1)(ka)=(hk-2)a.$$
    Otherwise, using repeated applications of \thref{a<b and c<d implies a+c<b+d} on the terms with nonzero coefficients in the equation for $c$, we get that
    $$c\prec (k-2)a+(h-1)(ka)=(hk-2)a.$$
    Hence, $(hk-1)a\notin hA$. By \thref{h<=i iff ha<=ia}, $(hk-2)a\prec (hk)a=h(ka)$. So, $hA\subseteq\br{ja\;\big|\; j\in[0,hk]-\br{hk-1}}$. Now, we will confirm that $\br{ja\;\big|\; j\in[0,hk]-\br{hk-1}}\subseteq hA$. Observe that
    \begin{alignat*}{2}
        (hk-i)a&=(k-i)a+(h-1)(ka)\in hA,
    \end{alignat*}
    for all $2\leq i\leq k$. It remains to show that $\br{ja\;\big|\; j\in\bb{0,hk-(k+1)}}=\br{ja\;\big|\; j\in\bb{0,(h-2)k+(k-1)}}\subseteq hA$. Let $c\in\br{ja\;\big|\; j\in\bb{0,(h-2)k+(k-1)}}$. Then, $c=ja$, for some $j\in\bb{0,(h-2)k+(k-1)}$. By the division algorithm, there exists $q\in\bb{0,h-2}$ and $r\in\bb{0,k-1}$ such that
    $$j=qk+r.$$
    In any case of $r$, we get
    $$c=ja=
    \begin{cases}
        (h-q-1)0+ra+q(ka)&\quad\text{if } 0\leq r\leq k-2,\\
        (h-q-2)0+a+(k-2)a+q(ka)&\quad\text{if } r=k-1.
    \end{cases}
    $$
    It now follows that $\br{ja\;\big|\; j\in[0,hk]-\br{hk-1}}\subseteq hA$. 
    
    Therefore, with both inclusions, we have that $hA=\br{ja\;\big|\; j\in[0,hk]-\br{hk-1}}$, which has size $hk$. By our arbitrary choice of $k\geq 3$, we can conclude that $\abs{hA}=hk$, for all $k\geq 3$.
\end{proof}
\begin{corollary}
    Let $G$ be a torsion-free additive abelian group with $a\in G$ such that $a\neq 0$. If $k\geq 3$, then $hk\in\cc{R}_{G}\pp{h,k}$.
\end{corollary}
\begin{proof}
    Since $G$ is a torsion-free additive abelian group, then, by \thref{torsion-free abelian implies ordered}, $G$ can be ordered. Thus, by \thref{A exists with |hA|=hk}, we can conclude that $hk\in\cc{R}_{G}\pp{h,k}$.
\end{proof}
As a consequence of these two results, we see that the missing interval can't be extended to the right. 

\begin{centering}
    \section*{Acknowledgments}
\end{centering}

\addcontentsline{toc}{section}{Acknowledgments}

A sincere thank you to Melvyn Nathanson for introducing me to this topic, and providing me with the necessary resources and guidance to confirm these results.

\newpage
\printbibliography[title={\centering References}]
\addcontentsline{toc}{section}{References}
\end{document}